\documentclass[10pt, reqno]{amsart}
\numberwithin{equation}{section}
\usepackage{amssymb, color, cite}
\usepackage{hyperref}
\usepackage{graphicx}

\newcommand{\qtq}[1]{\quad\text{#1}\quad}

\let\Re=\undefined\DeclareMathOperator*{\Re}{Re}

\DeclareMathOperator{\atan}{arctan}

\newcommand{\R}{\mathbb{R}}
\newcommand{\C}{\mathbb{C}}

\newcommand{\eps}{\varepsilon}

\newcommand{\supp}{\text{supp}}

\renewcommand{\H}{\mathcal{H}}

\newtheorem{theorem}{Theorem}[section]

\newtheorem{lemma}[theorem]{Lemma}

\newtheorem{corollary}[theorem]{Corollary}

\newtheorem{proposition}[theorem]{Proposition}

\theoremstyle{definition}
\newtheorem{definition}[theorem]{Definition}
\newtheorem{remark}[theorem]{Remark}

\theoremstyle{remark}

\begin{document}

\title{The scattering map determines the nonlinearity}

\author[R. Killip]{Rowan Killip}
\address{Department of Mathematics, UCLA}
\email{killip@math.ucla.edu}

\author[J. Murphy]{Jason Murphy}
\address{Department of Mathematics \& Statistics, Missouri S\&T}
\email{jason.murphy@mst.edu}

\author[M. Visan]{Monica Visan}
\address{Department of Mathematics, UCLA}
\email{visan@math.ucla.edu}

\maketitle

\begin{abstract} Using the two-dimensional nonlinear Schr\"odinger equation (NLS) as a model example, we present a general method for recovering the nonlinearity of a nonlinear dispersive equation from its small-data scattering behavior.  We prove that under very mild assumptions on the nonlinearity, the wave operator uniquely determines the nonlinearity, as does the scattering map.  Evaluating the scattering map on well-chosen initial data, we reduce the problem to an inverse convolution problem, which we solve by means of an application of the Beurling--Lax Theorem. 
\end{abstract}

\section{Introduction}

We consider two-dimensional nonlinear Schr\"odinger equations of the form
\begin{equation}\label{nls}
i\partial_t u + \Delta u = F(u),\quad (t,x)\in\R\times\R^2,
\end{equation}
where we regard the nonlinearity $F:\C\to\C$ as an unknown parameter.  We restrict to a class of equations that admit a small-data scattering theory and demonstrate that the scattering map uniquely determines the nonlinear term.

The precise assumptions we need for the nonlinearity are as follows: 
\begin{definition}[Admissible]\label{D:admissible} We call $F:\C\to\C$ \emph{admissible} if $F(u)=h(|u|^2)u$ for some $h:[0,\infty)\to\C$ with 
\[
h(0) = 0 \qtq{and} |h'(\lambda)| \lesssim1+\lambda^{\frac{p}{2}-1}
\]
for some $2\leq p<\infty$.  We call $p$ the \emph{growth parameter} of $F$. 
\end{definition}

If $F$ is admissible with growth parameter $p$, then
\[
|F(z)| \lesssim |z|^{3}+|z|^{p+1}\qtq{and}|F_z(z)|+|F_{\bar z}(z)|\lesssim |z|^2+|z|^p
\]
uniformly for $z\in\C$.  Comparing with the standard power-type NLS, for which $F(u)=|u|^p u$, we see that the definition of admissible nonlinearities covers the entire $L^2$-critical and $L^2$-supercritical range.

For admissible nonlinearities, we have a small-data scattering theory in $H^1$: 

\begin{theorem}[Small data scattering]\label{T:scattering} Let $F:\C\to\C$ be an admissible nonlinearity with growth parameter $p\geq2$.  Define $s_p=1-\frac{2}{p}$ and
\begin{equation}\label{ball}
B_\eta = \{f\in H^1: \|f\|_{H^{s_p}}<\eta\}.
\end{equation}
There exists $\eta>0$ sufficiently small so that  any initial data $u_0\in B_\eta$ leads to a unique global solution $u$ to \eqref{nls} satisfying
\begin{equation}\label{scattering-bounds}
\|u\|_{L_t^3 L_x^6(\R\times\R^2)}\lesssim\|u_0\|_{L^2}\qtq{and}
\|u\|_{L_t^{3p/2}L_x^{3p}(\R\times\R^2)} \lesssim \|u_0\|_{ \dot H^{s_p}}.
\end{equation}
This solution scatters in both time directions, that is, there exist (necessarily unique) $u_\pm\in H^1$ so that
\begin{equation}\label{scatter}
\lim_{t\to\pm\infty} \|u(t)-e^{it\Delta} u_\pm\|_{H^1}=0.
\end{equation}
Additionally, for any $u_-\in B_\eta$ there exists a unique global solution $u$ to \eqref{nls} and a unique $u_+\in H^1$ such that \eqref{scatter} holds.
\end{theorem}

The mapping $u_0\mapsto u_+$ described in Theorem~\ref{T:scattering} is known as the (forward) \emph{wave operator} and will be denoted $\Omega_F:B\to H^1$.  The mapping $u_-\mapsto u_+$ is known as the \emph{scattering map}, which we denote by $S_F:B\to H^1$.  If $\Omega_F:B\to H^1$ and $\Omega_{\tilde F}:\tilde B:\to H^1$ are the wave operators corresponding to a pair of admissible nonlinearities, then Theorem~\ref{T:scattering} guarantees that $B\cap \tilde B\neq\emptyset$.  This ensures that there is a common domain on which we may compare the scattering behaviors. 

Our main result asserts that knowledge of either the wave operator or the scattering map uniquely determines the nonlinearity in \eqref{nls}.

\begin{theorem}[Scattering determines the nonlinearity]\label{T} Let $F:\C\to\C$ and $\tilde F:\C\to\C$ be admissible, with potentially distinct growth parameters.  If $\Omega_F=\Omega_{\tilde F}$ or $S_F=S_{\tilde F}$ on $B\cap \tilde B$, then $F=\tilde F$.
\end{theorem}

There is a large body of literature concerning the recovery of the nonlinearity (as well as external potentials) from the scattering map in the setting of nonlinear dispersive equations (see e.g. \cite{SBUW, MorStr, CarGal, PauStr, Sasaki, Sasaki2, Sasaki3, SasakiWatanabe, Watanabe, Watanabe2, Watanabe3, Watanabe4, Weder1, Weder2, Weder3, Weder4, Weder5, Weder6}).  In general, these works either consider analytic nonlinearities or make other strong structural assumptions on the nonlinearity.  The work \cite{CarGal} provides an exhaustive treatment of the analytic case; other representative examples include the recovery of the coupling constant in a power-type nonlinearity \cite{MorStr}, the recovery of a Hartree potential \cite{Sasaki3}, or the recovery of an inhomogeneous coefficient in a nonlinearity of the form $q(x)|u|^p u$ \cite{Watanabe4}.  We were inspired to consider the problem discussed here by the recent work \cite{SBUW}, which established a result similar to Theorem~\ref{T} for nonlinear wave equations in three space dimensions.  In that paper the nonlinearity is assumed to be of quintic-type.  The authors of \cite{SBUW} employ techniques from microlocal analysis to study the propagation of singularities arising from nonlinear interactions, which in turn determine the higher order derivatives of the nonlinearity.

Compared to the previous literature, we work with very mild assumptions on the nonlinearity and prove that the \emph{entire} nonlinearity is determined by the small-data scattering behavior.  Our approach, which we describe below, is technically much simpler than the  analysis appearing in \cite{SBUW}. The simplicity of our arguments promises broad applicability.  To best present our method, we have chosen to focus on the concrete two-dimensional NLS problem laid out above. 

Let us now describe our strategy.  Our first observation is that it suffices to know the scattering behavior for a very narrow class of initial data, data for which the wave operator may be conflated with its Born approximation
\begin{align}
u_0\mapsto u_0 - \int_0^\infty e^{-it\Delta}F(e^{it\Delta}u_0)\,dt,
\end{align}
and likewise for the scattering map.

Taking this Born approximation for granted, we see that knowledge of the wave operator allows us to evaluate integrals of the form
\begin{align}
\int_0^\infty \langle e^{it\Delta}u_0 , F(e^{it\Delta}u_0)\rangle\,dt,
\end{align}
which may be interpreted as an inner product between the nonlinearity and the distribution function of the free solution $e^{it\Delta}u_0$; see Lemma~\ref{L:2}.

Pursuing this line of reasoning, we will ultimately be able to reduce the question of uniquely determining the nonlinearity to solving an inverse convolution problem.  Specifically, considering well-chosen Gaussian initial data, we will prove that if $\Omega_F=\Omega_{\tilde F}$, then
\begin{equation}\label{inverse-convolution}
\int_\R [G'(e^{-k})-\tilde G'(e^{-k})]e^{-k}w(k+\ell)\,dk = 0 \qtq{for all}\ell\in\R,
\end{equation}
where $G(|u|^2):=F(u)\bar u$ and $w$ is a weight related to the distribution function for the linear Schr\"odinger flow with Gaussian initial data.  The problem then reduces to showing that \eqref{inverse-convolution} implies that $G'\equiv \tilde G'$. Under mild hypotheses on the nonlinearity, this can be derived from Wiener's Tauberian Theorem.  To address the full range of admissible nonlinearities, however, we employ a theorem of Beurling and Lax characterizing shift-invariant subspaces of the Hardy space.  Here we take advantage of the fact that in two space dimensions, we are able to carry out explicit computations for $w$.  In particular, we prove that the Laplace transform of $w$ defines an outer function in the relevant half-plane.  This problem is complicated by the fact that $G'$ grows exponentially as $k\to-\infty$, while $w$ grows exponentially as $k\to+\infty$. 

The rest of this paper is organized as follows:  In Section~\ref{S:prelim} we introduce notation and collect basic lemmas. In Section~\ref{S:SDS}, we establish the small-data scattering theory for \eqref{nls} with admissible nonlinearities.  In Section~\ref{S:proof}, we reduce the proof of Theorem~\ref{T} to the inverse convolution problem described above, which we state as Theorem~\ref{P:key}. 
Section~\ref{S:key} is dedicated to the proof of Theorem~\ref{P:key}.  We first review the Beurling--Lax Theorem and relate this general result to the specific inverse convolution problem under consideration. We then demonstrate that the Laplace transform of $w$ is an outer function, which is precisely the input needed to apply the Beurling--Lax Theorem.  This part of the argument relies on an explicit computation of the Laplace transform, which is given in terms of the Gamma function.  With these ingredients in place, we complete the proof of Theorem~\ref{P:key}.  

In Section~\ref{S:special}, we show how additional restrictions on the nonlinearity greatly reduce the burden of understanding $w$.  Concretely, we show that one can recover polynomial-type nonlinearities without difficulty.

\subsection*{Acknowledgements} R. K. was supported by NSF grants DMS-1856755 and DMS-2154022.  J. M. was supported by a Simons Collaboration Grant. M. V. was supported by NSF grant DMS-2054194. 

%%%%%%%%%%%%%%%%%%%%%%%%%%%%%%
%%%%%%%%%%%%%%%%%%%%%%%%%%%%%%
\section{Preliminaries}\label{S:prelim}

We write $A\lesssim B$ to indicate that $A\leq CB$ for some $C>0$.  We indicate dependence on parameters via subscripts, e.g. $A\lesssim_u B$ means that $A\leq CB$ for some $C=C(u)$. If $A\lesssim B\lesssim A$, we write $A\approx B$.

We write $\langle\cdot,\cdot\rangle$ to denote the $L^2$ inner product. Given $q\in[1,\infty]$, we write $q'$ to denote the H\"older dual of $q$, that is, the solution to $\tfrac{1}{q}+\tfrac{1}{q'}=1$. 

We next record the standard Strichartz estimates for $e^{it\Delta}$ in the two-dimensional setting (see e.g. \cite{GinibreVelo}).  Recall that a pair $(q,r)\in(2,\infty]\times[2,\infty)$ is called \emph{Schr\"odinger admissible} in two space dimensions if $\tfrac{1}{q}+\tfrac{1}{r}=\tfrac{1}{2}$.

\begin{lemma}[Strichartz estimates] For any Schr\"odinger admissible pair $(q,r)$ and any $\varphi\in L^2$, 
\[
\|e^{it\Delta}\varphi\|_{L_t^q L_x^r(\R\times\R^2)}\lesssim \|\varphi\|_{L^2}. 
\]
Given an interval $I\ni 0$, Schr\"odinger admissible pairs $(q,r),(\tilde q,\tilde r)$, and $F\in L_t^{\tilde q'}L_x^{\tilde r'}(I\times\R^2)$,
\[
\biggl\| \int_0^t e^{i(t-s)\Delta}F(s)\,ds \biggr\|_{L_t^q L_x^r(I\times\R^2)}\lesssim \|F\|_{L_t^{\tilde q'}L_x^{\tilde r'}(I\times\R^2)}. 
\]
\end{lemma}

We will also use the following fractional calculus estimate from \cite{ChristWeinstein}.
\begin{lemma}[Fractional chain rule] Let $F:\C\to\C$ satisfy
\[
|F(u)-F(v)| \leq [K(u)+K(v)] |u-v|\qtq{for some} K:\C\to[0,\infty).
\]
For any $s\in(0,1)$, $r,r_1\in(1,\infty)$, and $r_2\in(1,\infty]$ satisfying $\tfrac{1}{r}=\tfrac{1}{r_1}+\tfrac{1}{r_2}$, we have
\[
\||\nabla|^s F(u)\|_{L^r} \lesssim \| K(u)\|_{L^{r_2}} \||\nabla|^s u\|_{L^{r_1}}.
\]
\end{lemma}

%%%%%%%%%%%%%%%%%%%%%%%%%%%%%%
%%%%%%%%%%%%%%%%%%%%%%%%%%%%%%
\section{Small data scattering}\label{S:SDS}

This section is dedicated to the proof of Theorem~\ref{T:scattering}.  This will be achieved via the standard contraction mapping argument using the Duhamel formulation of \eqref{nls}, namely,
\begin{equation}\label{duhamel}
u(t) = e^{it\Delta}u_0 - i \int_0^t e^{i(t-s)\Delta}F(u(s))\,ds,
\end{equation}
where $u_0=u|_{t=0}$.

For the construction of the scattering map, we use the analogous
\begin{equation}\label{duhamel-scatter}
u(t) = e^{it\Delta}u_- - i \int_{-\infty}^t e^{i(t-s)\Delta}F(u(s))\,ds.
\end{equation}

\begin{proof}[Proof of Theorem~\ref{T:scattering}] We begin with the construction of the solution.  All space-time norms will be taken over $\R\times\R^2$, unless indicated otherwise.

We will show that the map
\[
u\mapsto \Phi(u):= e^{it\Delta}u_0 - i\int_0^t e^{i(t-s)\Delta}F(u(s))\,ds
\]
is a contraction on the complete metric space $(Z,d)$, whenever $u_0\in B_\eta$ for $\eta$ sufficiently small (cf. \eqref{ball}). Here, 
\begin{align*}
Z:=\Bigl\{u:\R\times\R^2\to\C:\ \|u\|_X \leq 4C\|u_0\|_{\dot H^{s_p}},\quad &\|u\|_{L_t^\infty L_x^2 \cap L_t^3 L_x^6} \leq 4C\|u_0\|_{L^2}, \\
& \|\nabla u\|_{L_t^\infty L_x^2 \cap L_t^3 L_x^6}\leq 4C\|\nabla u_0\|_{L^2}\Bigr\},
\end{align*}
with
\[
\|u\|_X = \|u\|_{L_t^{3p/2}L_x^{3p}} + \| |\nabla|^{s_p} u\|_{L_t^3 L_x^6},
\]
and
\[
d(u,v) := \|u-v\|_{L_t^3 L_x^6}. 
\]
The constant $C>0$ in the definition of $Z$ is universal and encodes implicit constants appearing in the Sobolev embedding and Strichartz inequalities below. 

Using Sobolev embedding, Strichartz estimates, H\"older's inequality, the fractional chain rule, and the properties of $F$, for $u\in Z$ we estimate
\begin{align*}
\|\Phi(u)\|_X & \lesssim \|u_0\|_{\dot H^{s_p}} + \||\nabla|^{s_p} F(u)\|_{L_t^1 L_x^2} \\
& \lesssim \|u_0\|_{\dot H^{s_p}}  + \bigl[|u\|_{L_t^{3}L_x^{6}}^2+\|u\|_{L_t^{3p/2}L_x^{3p}}^p \bigr]\| |\nabla|^{s_p}u\|_{L_t^3 L_x^6}\\
& \lesssim \|u_0\|_{\dot H^{s_p}} + \bigl[\|u_0\|_{H^{s_p}}^{2}+\|u_0\|_{H^{s_p}}^{p}\bigr]\|u_0\|_{\dot H^{s_p}}.
\end{align*}
In particular, for $\eta$ sufficiently small we obtain
\[
\|\Phi(u)\|_X \leq 4C\|u_0\|_{\dot H^{s_p}}.
\]

Similarly,
\begin{align*}
\|\nabla \Phi(u)\|_{L_t^\infty L_x^2 \cap L_t^3 L_x^6} & \lesssim \|\nabla u_0\|_{L^2} + \|\nabla F(u)\|_{L_t^1 L_x^2} \\
& \lesssim \|\nabla u_0\|_{L^2} + \bigl[\|u\|_{L_t^{3}L_x^{6}}^2+\|u\|_{L_t^{3p/2}L_x^{3p}}^p \bigr]\|\nabla u\|_{L_t^3 L_x^6} \\
& \lesssim \|\nabla u_0\|_{L^2} + \bigl[\|u_0\|_{H^{s_p}}^{2}+\|u_0\|_{H^{s_p}}^p\bigr]\|\nabla u_0\|_{L^2},
\end{align*}
so that 
\[
\|\nabla \Phi(u)\|_{L_t^\infty L_x^2\cap L_t^3 L_x^6}\leq 4C\|\nabla u_0\|_{L^2},
\]
provided $\eta$ is chosen small enough. Parallel arguments show that
\[
\| \Phi(u)\|_{L_t^\infty L_x^2\cap L_t^3 L_x^6}\leq 4C\| u_0\|_{L^2},
\]
and so we conclude that $\Phi:Z\to Z$.

Next, for $u,v\in Z$ we may bound
\begin{align*}
\|\Phi&(u)-\Phi(v)\|_{L_t^3 L_x^6}  \\
&\lesssim \bigl[\|u\|_{L_t^{3}L_x^{6}}^2+\|u\|_{L_t^{3p/2}L_x^{3p}}^p +\|v\|_{L_t^{3}L_x^{6}}^2+\|v\|_{L_t^{3p/2}L_x^{3p}}^p\bigr]\|u-v\|_{L_t^3 L_x^6} \\
& \lesssim \bigl[\|u_0\|_{H^{s_p}}^2 + \|u_0\|_{H^{s_p}}^p\bigr]\|u-v\|_{L_t^3 L_x^6},
\end{align*}
which shows that $\Phi$ is a contraction if $\eta$ is sufficiently small.  

By the Banach fixed point theorem, we deduce that $\Phi$ has a unique fixed point in $Z$, which yields the desired solution $u$ to \eqref{nls} satisfying the bounds \eqref{scattering-bounds}. 

We next construct the asymptotic states $u_\pm$.  By time reversal symmetry, it suffices to establish scattering forward in time.  To this end, we fix $t>s>0$ and estimate as above to obtain
\begin{align*}
\|&e^{-it\Delta}u(t)-e^{-is\Delta}u(s)\|_{L_t^\infty H_x^1} \\
& \lesssim \bigl[\|u\|_{L_t^{3}L_x^{6}((s,t)\times\R^2)}^2+\|u\|_{L_t^{3p/2}L_x^{3p}((s,t)\times\R^2)}^p\bigr]\|\langle\nabla\rangle u\|_{L_t^3 L_x^6((s,t)\times\R^2)}, 
\end{align*}
which converges to zero as $s,t\to\infty$ by \eqref{scattering-bounds} and the monotone convergence theorem.  It follows that $\{e^{-it\Delta}u(t)\}$ is Cauchy in $H^1$ as $t\to\infty$ and so converges to a unique limit $u_+\in H^1$.

The construction of the full scattering map is completely analogous, using \eqref{duhamel-scatter} instead of \eqref{duhamel} to construct the solution. \end{proof}

\begin{remark} The Duhamel formula \eqref{duhamel} for $u$ shows that the wave and scattering operators satisfy
\begin{equation}\label{form-of-S}
\begin{aligned}
\Omega_F(u_0)&=u_0-i\int_0^{\infty} e^{-it\Delta}F(u(t))\,dt,\\
S_F(u_-)& = u_- - i\int_{-\infty}^\infty e^{-it\Delta}F(u(t))\,dt.
\end{aligned}
\end{equation}
\end{remark}

%%%%%%%%%%%%%%%%%%%%%%%%%%%%%%
%%%%%%%%%%%%%%%%%%%%%%%%%%%%%%
\section{Reduction to an inverse convolution problem}\label{S:proof}

Throughout this section, we use the notation
\begin{equation}\label{def:G}
G(|u|^2) := F(u)\bar u,
\end{equation}
where $F$ is an admissible nonlinearity, and we consider Gaussian initial data of the form
\begin{equation}\label{gaussian}
u_0^\sigma(x)=Ae^{-|x|^2/4\sigma^2}\qtq{with} A,\sigma>0. \end{equation}

The majority of this section is devoted to the proof of Proposition~\ref{C:12}, which reduces the proof of Theorem~\ref{T} to the consideration of an inverse convolution problem.  The resolution of this convolution problem is stated as Theorem~\ref{P:key} below.  Using Proposition~\ref{C:12} and Theorem~\ref{P:key}, we complete the proof of Theorem~\ref{T} at the end of this section. The proof of Theorem~\ref{P:key} will then be given in Section~\ref{S:key}. 

\begin{lemma}\label{L:1} Let $F,\tilde F$ be admissible and let $\Omega_F,S_F:B\to H^1$ and $\Omega_{\tilde F},S_{\tilde F}:\tilde B\to H^1$ be the corresponding wave and scattering operators. Let $u_0^\sigma$ be as in \eqref{gaussian}.  Then 
\begin{itemize}
\item[(i)] $u_0^\sigma\in B\cap \tilde B$ for  sufficiently small $\sigma$.
\item[(ii)] If $\Omega_F(u_0^\sigma)=\Omega_{\tilde F}(u_0^\sigma)$, then 
\begin{equation}\label{G-id}
\biggl| \int_0^\infty\int_{\R^2} G(|e^{it\Delta}u_0^\sigma|^2)-\tilde G(|e^{it\Delta}u_0^\sigma|^2)\,dx\,dt\biggr| \lesssim_A\sigma^6.
\end{equation}
\item[(iii)] If $S_F(u_0^\sigma)=S_{\tilde F}(u_0^\sigma)$, then \eqref{G-id} holds with the time integral taken over $\R$. 
\end{itemize}
\end{lemma}

\begin{proof} Item (i) follows from the fact that for any $s\geq 0$, 
\begin{equation}\label{u0-bd}
\|u_0^\sigma\|_{\dot H^s(\R^2)} \lesssim_s A \sigma^{1-s},
\end{equation}
along with the fact that $B,\tilde B$ are of the form \eqref{ball}. 

We turn to (ii).  Writing $u,\tilde u$ for the solutions to \eqref{nls} with nonlinearities $F,\tilde F$ and initial data $u_0^\sigma$, we use \eqref{form-of-S} to write
\begin{align*}
\langle \Omega_F(u_0^\sigma)-\Omega_{\tilde F}(u_0^\sigma),u_0^\sigma\rangle & = -i\int_0^\infty \langle e^{-it\Delta}[F(u(t))-\tilde F(\tilde u(t))],u_0^\sigma\rangle\,dt \\
& = -i\int_0^\infty\int_{\R^2} \bigl[G(|e^{it\Delta}u_0^\sigma|^2)-\tilde G(|e^{it\Delta}u_0^\sigma|^2)\bigr]\,dx\,dt \\
& \quad -i\int_0^\infty \langle F( u(t))-F(e^{it\Delta}u_0^\sigma),e^{it\Delta}u_0^\sigma\rangle\,dt \\
& \quad +i\int_0^\infty \langle \tilde F(\tilde u(t)) - \tilde F(e^{it\Delta}u_0^\sigma),e^{it\Delta}u_0^\sigma\rangle\,dt.
\end{align*}
By assumption, $\Omega_F(u_0^\sigma)=\Omega_{\tilde F}(u_0^\sigma)$, so item (ii) will follow once we prove that
\begin{equation}\label{bound-error}
\biggl|\int_0^\infty \langle F( u(t))-F(e^{it\Delta}u_0^\sigma),e^{it\Delta}u_0^\sigma\rangle\,dt\biggr| \lesssim_A \sigma^6
\end{equation}
for any admissible $F$ (including $\tilde F$). 

To establish \eqref{bound-error}, we let
\[
N(t) := u(t)-e^{it\Delta}u_0^\sigma = -i\int_0^\infty e^{i(t-s)\Delta}F(u(s))\,ds
\]
and let $p$ denote the growth parameter of $F$.  Using Strichartz, Sobolev embedding, and \eqref{scattering-bounds}, we may bound
\begin{align*}
\biggl| \int_0^\infty &\langle F(u(t))-F(e^{it\Delta}u_0^\sigma),e^{it\Delta}u_0^\sigma\rangle\,dt \biggr| \\
& \lesssim \|e^{it\Delta} u_0^\sigma\|_{L_t^3 L_x^6}\| F(u(t))-F(e^{it\Delta}u_0^\sigma)\|_{L_t^{3/2}L_x^{6/5}}  \\
& \lesssim \|u_0^\sigma\|_{L^2}\| N(t)\|_{L_t^\infty L_x^2}\\
& \quad\times\bigl[ \|u^2\|_{L_t^{3/2}L_x^3} + \|[e^{it\Delta}u_0^\sigma]^2\|_{L_t^{3/2}L_x^3}+\|u^p\|_{L_t^{3/2}L_x^3} + \|[e^{it\Delta}u_0^\sigma]^p\|_{L_t^{3/2}L_x^3}\bigr] \\
& \lesssim \|u_0^\sigma\|_{L^2}\bigl[ \|u^{3}\|_{L_t^1 L_x^2}+\| u^{p+1}\|_{L_t^1 L_x^2}\bigr]\\
&\quad \times \bigl[\|u\|_{L_t^{3}L_x^{6}}^2+\|e^{it\Delta}u_0^\sigma\|_{L_t^{3}L_x^{6}}^2+ \|u\|_{L_t^{3p/2}L_x^{3p}}^p+\|e^{it\Delta}u_0^\sigma\|_{L_t^{3p/2}L_x^{3p}}^p  \bigr] \\
& \lesssim \|u_0^\sigma\|_{L^2}\|u\|_{L_t^3 L_x^6}\bigl[\|u\|_{L_t^{3}L_x^{6}}^2+\|u\|_{L_t^{3p/2}L_x^{3p}}^p \bigr] \bigl[ \|u_0^\sigma\|_{L^2}^2+\|u_0^\sigma\|_{\dot H^{1-2/p}}^p \bigr]\\
& \lesssim \|u_0^\sigma\|_{L^2}^2\bigl[ \|u_0^\sigma\|_{L^2}^2+\|u_0^\sigma\|_{\dot H^{1-2/p}}^p\bigr]^2,
\end{align*}
where we used \eqref{scattering-bounds} in the last two steps.  The estimate \eqref{bound-error} now follows from \eqref{u0-bd}.

Part (iii) follows from a direct recapitulation of the proof of (ii), using the second formula in \eqref{form-of-S}. \end{proof}

\begin{lemma}\label{L:2} Let $F$ be an admissible nonlinearity and let $u_0^\sigma$ be as in \eqref{gaussian}.  Define $G$ as in \eqref{def:G} and let
\begin{equation}\label{def:H}
H(k):=G'(e^{-k})e^{-k}.
\end{equation}
Then
\begin{equation}\label{G-identity}
\int_0^\infty \int_{\R^2} G(|e^{it\Delta}u_0^\sigma|^2)\,dx\,dt = \tfrac{4\pi}{9}\sigma^4\int_\R H(k) w(k+2\log A)\,dk,
\end{equation}
where
\begin{equation}\label{def:w1}
w(k):= \bigl[(e^k-1)^{\frac32}+6(e^k-1)^{\frac12}-6\atan\bigl((e^k-1)^{\frac12}\bigr)\bigr]\chi_{(0,\infty)}(k). 
\end{equation}
If the time integral in \eqref{G-identity} is extended to all of $\R$, the right-hand side doubles. 
\end{lemma}

\begin{proof} We first consider $\sigma=1$; for notational simplicity we write $u_0$ for $u_0^1$. Using the layer cake decomposition, we may write
\begin{align*}
\int_0^\infty&\int_{\R^2}G(|e^{it\Delta}u_0|^2)\,dx\,dt= \int_0^\infty G'(\lambda)\,\bigl|\{(t,x)\in(0,\infty)\times\R^2: |e^{it\Delta}u_0|^2>\lambda\}\bigr|\,d\lambda.
\end{align*}
By direct computation (see e.g. \cite[Equation~(2.4)]{Oberwolfach}), we have
\[
e^{it\Delta}u_0(x) = \tfrac{A}{1+it} \exp\bigl\{-\tfrac{|x|^2}{4(1+it)}\bigr\}.
\]
In order that $|e^{it\Delta}u_0(x)|^2>\lambda$ it is necessary that $\lambda<A^2$, in which case the inequality holds at space-time points where
\[
\lambda(1+t^2)<A^2 \qtq{and} |x|< \bigl[2(1+t^2)\log\bigl(\tfrac{A^2}{\lambda(1+t^2)}\bigr)\bigr]^{\frac{1}{2}}.
\]
Thus, integrating by parts, we find that
\begin{align*}
\int_0^\infty\int_{\R^2} G(|e^{it\Delta}u_0|^2)\,dx\,dt  &= \pi\int_0^{A^2} G'(\lambda)\int_0^{[A^2\lambda^{-1}-1]^{\frac12}}2(1+t^2)\log\bigl(\tfrac{A^2}{\lambda(1+t^2)}\bigr)\,dt \,d\lambda \\
& = \pi\int_0^{A^2}G'(\lambda)\int_0^{[A^2\lambda^{-1}-1]^{\frac12}}(t+\tfrac13t^3)\tfrac{4t}{1+t^2}\,dt\,d\lambda \\
& = \tfrac{4\pi}{9}\int_0^{A^2} G'(\lambda)w_0(\tfrac{\lambda}{A^2})\,d\lambda,
\end{align*}
where
\[
w_0(\lambda):=(\lambda^{-1}-1)^{\frac32}+6(\lambda^{-1}-1)^{\frac12}-6\atan\bigl((\lambda^{-1}-1)^{\frac12}\bigr). 
\]
Applying the change of variables $\lambda = e^{-k}$, we obtain 
\begin{align*}
\int_0^\infty \int_{\R^2} G(|e^{it\Delta}u_0|^2)\,dx\,dt  & = \tfrac{4\pi}{9}\int_{-2\log A}^\infty G'(e^{-k})e^{-k} w_0(e^{-k-2\log A}) \,dk \\
& = \tfrac{4\pi}{9}\int_{\R} H(k) w(k+2\log A)\,dk,
\end{align*}
where $H(\cdot)$ and $w(\cdot)$ are as in \eqref{def:H} and \eqref{def:w1}. This yields \eqref{G-identity} with $\sigma=1$.

To treat the case of general $\sigma>0$, we note that $u_0^\sigma(x)=u_0^1(\sigma^{-1}x)$, which implies
\[
e^{it\Delta}u_0^\sigma(x)= [e^{i\sigma^{-2}t\Delta}u_0^1](\sigma^{-1}x). 
\]
Thus, by a change of variables,
\[
\int_0^\infty \int_{\R^2} G(|e^{it\Delta}u_0^\sigma|^2)\,dx\,dt = \sigma^4 \int_0^\infty \int_{\R^2} G(|e^{it\Delta}u_0^1|^2)\,dx\,dt,
\]
which yields \eqref{G-identity}.

The final claim of the lemma follows by repeating the previous argument or, more simply, by exploiting time-reversal symmetry. \end{proof}

With Lemmas~\ref{L:1} and~\ref{L:2} in place, we are now in a position to reduce the proof of Theorem~\ref{T} to the consideration of an inverse convolution problem.

\begin{proposition}\label{C:12} Let $F,\tilde F$ be admissible and let $\Omega_F,S_F:B\to H^1$ and $\Omega_{\tilde F},S_{\tilde F}:\tilde B\to H^1$ be the corresponding wave and scattering operators. Define $H,\tilde H$ as in \eqref{def:H} and $w$ as in \eqref{def:w1}.  
If $\Omega_F=\Omega_{\tilde F}$ or $S_F=S_{\tilde F}$ on $B\cap\tilde B$, then
\[
\int_\R [H(k)-\tilde H(k)]w(k+\ell)\,dk = 0 \qtq{for all}\ell\in\R. 
\]
\end{proposition}

\begin{proof} It suffices to consider the case $\Omega_F=\Omega_{\tilde F}$.  The case $S_F=S_{\tilde F}$ follows in an identical fashion.

Fix $A>0$ and define $u_0^\sigma$ as in \eqref{gaussian}. Combining Lemmas~\ref{L:1} and~\ref{L:2}, we find that for all sufficiently small $\sigma$ we have 
\begin{align*}
\biggl|\tfrac{4\pi}{9}\int_\R& [H(k)-\tilde H(k)]w(k+2\log A)\,dk \biggr| \\
& = \sigma^{-4}\biggl| \int_0^\infty \int_{\R^2} G(|e^{it\Delta}u_0^\sigma|^2)-\tilde G(|e^{it\Delta}u_0^\sigma|^2)\,dx\,dt \biggr| \lesssim_A \sigma^2.
\end{align*}
As the left-hand side is independent of $\sigma$, the result follows by sending $\sigma\to0$. 
\end{proof}

The last ingredient in the proof of Theorem~\ref{T} is the solution of this inverse convolution problem.

\begin{theorem}\label{P:key} Let $F,\tilde F$ be admissible.  Define $H,\tilde H$ as in \eqref{def:H} and $w$ as in \eqref{def:w1}.  Given $\ell_0\in\R$, if
\begin{equation}\label{orthog-l0}
\int_\R [H(k)-\tilde H(k)]w(k+\ell)\,dk = 0 \qtq{for all}\ell\leq\ell_0,
\end{equation}
then $F(u)=\tilde F(u)$ for all $|u|\leq e^{\frac12\ell_0}$.
\end{theorem}

We will prove Theorem~\ref{P:key} in the next section.  For now, let us see it implies Theorem~\ref{T}.

\begin{proof}[Proof of Theorem~\ref{T}] If $F,\tilde F$ are admissible and the scattering data agree on $B\cap \tilde B$, then Proposition~\ref{C:12} and Theorem~\ref{P:key} imply $F(u)=\tilde F(u)$ for all $u\in\C$.\end{proof}

%%%%%%%%%%%%%%%%%%%%%%
\section{Proof of Theorem~\ref{P:key}}\label{S:key}

Our proof of Theorem~\ref{P:key} relies on the Beurling--Lax Theorem, which tells us when the span of the (right) translates of a function in $L^2([0,\infty))$ are dense in this space.  After discussing this theorem, we demonstrate that the necessary (and sufficient) condition is satisfied for the specific weight
\begin{equation}\label{def:w}
w(k):=\bigl[(e^k-1)^{\frac32}+6(e^k-1)^{\frac12}-6\atan\bigl((e^k-1)^{\frac12}\bigr)\bigr]\chi_{(0,\infty)}(k) 
\end{equation}
appearing in Theorem~\ref{P:key}; this allows us to complete the proof of Theorem~\ref{P:key}.

The following theorem, due to Lax \cite{MR0105620}, characterizes shift-invariant subspaces of the Hardy space $\H^2(\{\Re z>0\})$.  Here shift-invariance of a closed subspace $\mathcal{M}$ refers to the fact that $e^{-az}F(z)\in \mathcal{M}$ whenever $F\in \mathcal{M}$ and $a>0$.  The result relies on the inner/outer factorization on Hardy spaces; for a textbook presentation, see \cite[Chapter~5]{Hoffman}.

\begin{theorem}[Lax]\label{T:BL} If $\mathcal{M}$ is a closed, shift-invariant subspace of $\H^2$, then there exists an inner function $\theta$ such that $\mathcal{M}=\theta\H^2$. 
\end{theorem}
The analogous result for $\H^2(\mathbb{D})$ was established by Beurling \cite{MR0027954}; in fact, the half-plane case can be deduced from the disk case via conformal mapping, as demonstrated in \cite{Hoffman}.

We will use Theorem~\ref{T:BL} to prove the following corollary.  

\begin{corollary}\label{C:BL} Fix $v\in L^2(\R)$ with $\supp(v)\subseteq [0,\infty)$. Suppose the Laplace transform
\begin{align}\label{Big F}
V(z)= \mathcal{L}v(z) :=\int_0^\infty e^{-kz} v(k)\,dk
\end{align}
defines an outer function on the half-plane $\{\Re z >0\}$.  If $f\in L^2([0,\infty))$ satisfies
\begin{equation}\label{BL-orthog}
\int_0^\infty v(k-a)f(k)\,dk = 0 \qtq{for all}a\geq 0,
\end{equation}
then $f\equiv 0$.
\end{corollary}

\begin{proof} We first show that $\mathcal{M}=\overline{\text{span}\{e^{-az}V(z):a\geq0\}}$ is all of $\H^2$.  The proof of this fact is the same as that of \cite[Corollary~II.7.3]{MR2261424}, which treated the case of the disk.

We first note that $\mathcal{M}$ is a closed, shift-invariant subspace, and hence by Theorem~\ref{T:BL} we know that $\mathcal{M} = \theta\H^2$ for some inner function $\theta$. As $V\in \mathcal{M}$, we deduce that
\[
V=\theta U \qtq{for some}U\in\H^2.
\]
We write the inner/outer factorization of $U$ as $U=\vartheta O$.  Using the fact that $V$ is outer, while $\theta$ and $\vartheta$ (and hence $\theta\vartheta$) are inner, we can use the uniqueness of the inner/outer factorization deduce that $\theta\vartheta\equiv 1$.  This in turn guarantees that $\theta$ and $\vartheta$ are constant.  In particular, $\theta\H^2=\H^2$, yielding $\mathcal{M}=\H^2$.

To complete the proof, we observe that \eqref{BL-orthog} and Plancherel imply that $\mathcal{L}f\in \mathcal{M}^\perp=\{0\}$, which in turn guarantees $f\equiv 0$.  \end{proof}

Evidently, it is convenient to have a simple test to see if $V$ is outer.  The following suffices for our purposes: 

\begin{lemma}\label{L:outer-condition} Suppose $V\in\H^2$ extends continuously from $\{\Re z>0\}$ to $z\in i\R\backslash\{0\}$ and that there exists $\eps>0$ and $\beta\geq 1$ such that
\begin{equation}\label{V-bds}
\eps|z+1|^{-\beta} \leq |V(z)|\leq \eps^{-1}|z|^{-1}\qtq{for all}\Re z>0.
\end{equation}
Then $V$ is an outer function.
\end{lemma}

\begin{proof} We conformally transport the problem to the disk $\mathbb{D}$ via the M\"obius transformation $M(z):= \tfrac{1-z}{1+z}$, which maps $\mathbb{D}$ to $\{\Re z>0\}$ and respects the classes of inner and outer functions.  By \cite[Corollary~II.4.7]{MR2261424}, $V$ is an outer function if $V\circ M,\ 1/(V\circ M)$ belong to the Hardy space $\H^q(\mathbb{D})$ for some $q\in(0,\infty]$. Using \eqref{V-bds}, we see that 
\[
\sup_{r\in[0,1)} \int_{-\pi}^\pi \Bigl| V\bigl(\tfrac{1-re^{i\theta}}{1+re^{i\theta}}\bigr) \Bigr|^q + \Bigl| V\bigl(\tfrac{1-re^{i\theta}}{1+re^{i\theta}}\bigr) \Bigr|^{-q} \,d\theta < \infty 
\]
for any $0<q<\frac1\beta$.\end{proof}

\begin{proposition}\label{P:Nonvanishing} Let $w$ be as in \eqref{def:w}.  The function 
\begin{equation}\label{def:W}
W(z):=\int_0^\infty e^{-kz}w(k)\,dk,
\end{equation}
initially defined for $\Re z>\tfrac32$, admits a meromorphic extension to $\C$ given by
\[
W(z) = 9\sqrt{\pi}\ \frac{\Gamma(z+\frac12)}{\Gamma(z+1)}\ \frac{z-1}{z(2z-1)(2z-3)}.
\]
In particular, $W$ is outer in $\{\Re z>\tfrac74\}$, where it satisfies the bounds
\begin{equation}\label{W-lb}
|W(z)|\approx |z|^{-\frac52}.
\end{equation}

\end{proposition}

\begin{proof} We begin with a special case of Euler's Beta integral.  Given $-1<\alpha<\Re z$, the change of variables $u=e^{-k}$ yields
\begin{align*}
\int_0^\infty e^{-kz}(e^k-1)^\alpha\,dk &= \int_0^1 u^{z-\alpha-1}(1-u)^\alpha\,du =\frac{\Gamma(z-\alpha)\Gamma(1+\alpha)}{\Gamma(z+1)}.
\end{align*} 
Thus, using the identity 
\[
e^{-kz}=-\tfrac{1}{z}\tfrac{d}{dk}e^{-kz}
\]
to integrate by parts in the $\atan$ term, along with the identities $\Gamma(\tfrac12)=\sqrt{\pi}$ and $\Gamma(z+1)=z\Gamma(z)$, we obtain
\begin{align*}
W(z) & = \int_0^\infty e^{-kz}\bigl[(e^k-1)^{\frac32}+6(e^k-1)^{\frac12}-6\atan\bigl((e^k-1)^{\frac12}\bigr)\bigr]\,dk \\
& = \int_0^\infty e^{-kz}\bigl[(e^k-1)^{\frac32}+6(e^k-1)^{\frac12}-\tfrac{3}{z}(e^k-1)^{-\frac12}\bigr] \,dk \\
& = \frac{1}{\Gamma(z+1)}\biggl[\Gamma(z-\tfrac32)\Gamma(\tfrac52)+6\Gamma(z-\tfrac12)\Gamma(\tfrac32)-\tfrac{3}{z}\Gamma(z+\tfrac12)\Gamma(\tfrac12)\biggr] \\
& = \frac{\sqrt{\pi}\ \Gamma(z+\frac12)}{\Gamma(z+1)}\biggl[ \frac{\frac32\cdot\frac12}{(z-\frac32)(z-\frac12)}+\frac{3}{z-\frac12}-\frac{3}{z}\biggr] \\
& = 9\sqrt{\pi}\ \frac{\Gamma(z+\frac12)}{\Gamma(z+1)}\ \frac{z-1}{z(2z-1)(2z-3)}
\end{align*}
for all $z\in\C$ with $\Re z>\frac32$. The extension to $\C$ now follows from analytic continuation. 

Regarding \eqref{W-lb}, we first recall
\cite[Theorem~A, p. 68]{Rademacher}, which says that
\[
\biggl|\frac{\Gamma(s+c)}{\Gamma(s)}\biggr|\leq |s|^c\qtq{for} c\in[0,1]\qtq{and}\Re s \geq \tfrac12(1-c).
\]
In particular, 
\[
|z+\tfrac12|^{-\frac12}\leq \biggl| \frac{\Gamma(z+\frac12)}{\Gamma(z+1)}\biggr| \leq \frac{|z+1|^{\frac12}}{|z+\frac12|\ } \qtq{provided}\Re z>-\tfrac14,
\]
which leads to \eqref{W-lb}. By Lemma~\ref{L:outer-condition}, this guarantees that $W$ is outer.\end{proof}

We are now ready to complete the proof of Theorem~\ref{P:key}.

\begin{proof}[Proof of Theorem~\ref{P:key}] Fix $\ell_0\in\R$.  Using \eqref{orthog-l0} and a change of variables, we may derive that 
\[
\int_0^\infty e^{\frac74k}[H(k-\ell_0)-\tilde H(k-\ell_0)] e^{-\frac74(k-a)}w(k-a)\,dk = 0 \qtq{for all}a\geq 0. 
\]
By the assumptions on $F,\tilde F$ and the definition of $H,\tilde H$ (cf. \eqref{def:H}), we have
\begin{align*}
|e^{\frac74k}[H(k-\ell_0)-\tilde H(k-\ell_0)]| &\lesssim e^{2\ell_0}e^{-\frac14 k}+e^{\frac{p+2}{2}\ell_0}e^{-[\frac{p}{2}-\frac34]k}+e^{\frac{\tilde p+2}{2}\ell_0}e^{-[\frac{\tilde p}{2}-\frac34]k} \\
& \lesssim_{\ell_0} e^{-\frac14 k}+e^{-[\frac{p}{2}-\frac34]k}+e^{-[\frac{\tilde p}{2}-\frac34]k}.
\end{align*}
As $p,\tilde p\in[2,\infty)$, we have
\[
e^{\frac74k}[H(k-\ell_0)-\tilde H(k-\ell_0)]\in L^2([0,\infty)).
\]
Similarly, by the definition of $w$ (cf. \eqref{def:w}), 
\[
|e^{-\frac74k}w(k)| \lesssim e^{-\frac14 k} \in L^2([0,\infty)). 
\]
Thus we may apply Corollary~\ref{C:BL} to deduce that 
\begin{equation}\label{eq-restrict}
H(k-\ell_0)=\tilde H(k-\ell_0) \qtq{for all}k\geq 0,
\end{equation}
provided we can verify that
\[
V(z) = \int_0^\infty e^{-zk} e^{-\frac74k}w(k)\,dk = W(z+\tfrac74)
\]
is an outer function on $\{\Re z>0\}$, where $W$ is as in \eqref{def:W}.  This follows from Proposition~\ref{P:Nonvanishing}. 

Now observe that \eqref{eq-restrict} implies that $G(\lambda)=\tilde G(\lambda)$ for all $\lambda\leq e^{\ell_0}$ and so $F(u)=\tilde F(u)$ for all $|u|\leq e^{\frac12\ell_0}$. \end{proof}

\section{Some special cases}\label{S:special} In this section, we discuss a few special cases inspired by previous works.  These require considerably less detailed information about the weight $w$, and consequently can be more easily adapted to other models. 

We first consider the case when $F$ and $\tilde F$ are of `generalized polynomial' type.  In particular, our first result is an extension of the results appearing in works such as \cite{MorStr, CarGal, PauStr}, in which it is shown that in the case $F(u)=\lambda |u|^{2k} u$ (with $k$ a positive integer), the scattering map uniquely determines $k$ and $\lambda$.  As our argument will only make use of the positivity of the weight $w$, it extends readily to all dimensions.

\begin{theorem}[Generalized polynomial case]\label{T:polynomial} Suppose
\[
F(u) = \sum_{p\in D} a_p |u|^p u \qtq{and} \tilde F(u)=\sum_{q\in \tilde D} \tilde a_q |u|^q u
\]
where $D,\tilde D$ are finite subsets of $[2,\infty)$ and $a_p,\tilde a_q\in\C$.  Let $\Omega_F,S_F:B\to H^1$ and $\Omega_{\tilde F},S_{\tilde F}:\tilde B\to H^1$ be the associated wave and scattering operators.  If $\Omega_F=\Omega_{\tilde F}$ or $S_F=S_{\tilde F}$ on $B\cap \tilde B$, then $F=\tilde F$.
\end{theorem}

\begin{proof} We first note that $F,\tilde F$ are admissible in the sense of Definition~\ref{D:admissible}, so that Theorem~\ref{T:scattering} applies and yields the existence of the wave and scattering operators.  The growth parameters of $F,\tilde F$ correspond to the largest elements of $D,\tilde D$. 

From the explicit form of $F$ and $\tilde F$ and the definition of $H,\tilde H$ (see \eqref{def:H}), we find
\[
H(k)-\tilde H(k) = \sum_{r\in E} b_r e^{-\frac{r+2}{2}k}
\]
for $E=D\cup\tilde D$ and coefficients $b_r\in\C$.  To show that $F=\tilde F$, we will show that $b_r=0$ for each $r\in E$.

In view of Proposition~\ref{C:12}, if $\Omega_F=\Omega_{\tilde F}$ or $S_F = S_{\tilde F}$ on $B\cap \tilde B$, 
\begin{equation}\label{7:00}
\sum_{r\in E}b_r \biggl[ \int_0^\infty e^{-\frac{r+2}{2}k}w(k)\,dk\biggr] e^{\frac{r+2}{2}\ell} = 0 \qtq{for all}\ell\in\R,
\end{equation}
with $w$ as in \eqref{def:w1}.  As $w$ is defined as the measure of superlevel sets, it is a nonnegative function.  This can also be verified directly from the explicit formula \eqref{def:w1}.  As $w$ is not identically zero, it follows that the coefficients in square brackets in \eqref{7:00} are always positive.  Thus, by the linear independence of the functions $\ell\mapsto e^{\ell(r+2)/2}$, we see that $b_r\equiv 0$. 
\end{proof}

We next consider the case when $F$ and $\tilde F$ have the same size as a common single-power nonlinearity.

\begin{definition}[$p$-admissible] Let $2\leq p<\infty$.  We call $F:\C\to\C$ \emph{$p$-admissible} if $F(u)=h(|u|^2)u$ for some $h:[0,\infty)\to\C$ with
\[
h(0)= 0 \qtq{and} |h'(\lambda)| \approx \lambda^{\frac{p}{2}-1}. 
\]
\end{definition}

In particular, if $F$ is $p$-admissible then it is admissible in the sense of Definition~\ref{D:admissible} and we have 
\[
|F(u)|\approx|u|^{p+1}.
\]
Our consideration of this case is inspired by \cite{SBUW}, which treated the nonlinear wave equation in three dimensions with quintic-type nonlinearities. 

\begin{theorem}[Single-power case]\label{T:single-power} Suppose $F$ and $\tilde F$ are $p$-admissible with $p\geq 2$ and let $\Omega_F,S_F:B\to H^1$ and $\Omega_{\tilde F},S_{\tilde F}:\tilde B\to H^1$ be the associated wave and scattering operators.  If $\Omega_F=\Omega_{\tilde F}$ or $S_F=S_{\tilde F}$ on $B\cap \tilde B$, then $F=\tilde F$.
\end{theorem}

\begin{proof} By Proposition~\ref{C:12}, we have
\[
\int_{\R}[H(k)-\tilde H(k)]w(k+\ell)\,dk =0 \qtq{for all}\ell\in\R,
\]
with $H,\tilde H$ as in \eqref{def:H} and $w$ as in \eqref{def:w1}.  We rewrite this as 
\[
\int_\R \ e^{k\frac{p+2}{2}}[H(k)-\tilde H(k)]\cdot  e^{-(k+\ell)\frac{p+2}{2}}w(k+\ell)\,dk = 0 \qtq{for all}\ell\in\R
\]
and note that as $F$ and $\tilde F$ are both $p$-admissible with $p>1$,
\[
e^{k\frac{p+2}{2}}[H(k)-\tilde H(k)]\in L^\infty \qtq{and} e^{-k\frac{p+2}{2}}w(k)\in L^1. 
\]
In particular, using Wiener's Tauberian Theorem \cite{Wiener}, we may deduce that $H=\tilde H$ (and hence $F=\tilde F$), provided
\begin{equation}\label{nonzero-special}
W(\tfrac{p+2}{2}+i\xi):=\int_0^\infty e^{-k[\frac{p+2}{2}+i\xi]}w(k)\,dk \neq 0 \qtq{for all}\xi\in\R. 
\end{equation}
By Proposition~\ref{P:Nonvanishing}, $W(z)\neq 0$ whenever $\Re z>\tfrac32$, which includes \eqref{nonzero-special} as a special case.%
\end{proof} 

It is truly necessary to verify \eqref{nonzero-special} for \emph{all} values of $\xi\in\R$, for otherwise $H-\tilde H$ could be a sinusoid of the corresponding frequency.

\end{document}